\documentclass{article}
\usepackage{graphicx} 
\usepackage{amssymb}
\usepackage{amsmath}
\usepackage{amsthm}

\newtheorem{theorem}{Theorem}[section]
\newtheorem{lemma}[theorem]{Lemma}

\DeclareMathOperator{\Gr}{Gr}
\DeclareMathOperator{\PGL}{PGL}
\DeclareMathOperator{\GL}{GL}
\DeclareMathOperator{\Spec}{Spec}
\DeclareMathOperator{\stab}{stab}

\title{On the codimension 1 PGL(3) orbit closures in $\text{Gr}(3,6)$}
\author{Tanav Choudhary}
\date{June 2025}

\begin{document}

\maketitle

\begin{abstract}
The projective linear group $\PGL(3)$ naturally acts on the Grassmannian $\Gr(3, V_2)$ of $3$-dimensional subspaces of the vector space $V_2$ of homogeneous conics in 3 variables. It was proved by Abdallah, Emsalem and Iarrobino in 2021 that this action has a one-parameter family of orbits along with 14 special orbits. The codimension 1 orbits of this action consist of the entire one-parameter family of orbits, along with 2 of the 14 special orbits. In this paper, we calculate the classes of the codimension 1 orbit closures in the Chow ring of $\Gr(3, V_2)$. 
\end{abstract}

\section{Introduction and Set-up}

Let $k$ be an algebraically closed field of characteristic zero. For each $i$, let $V_i$ be the vector space of homogeneous polynomials of degree $i$ in 3 variables over $k$. In particular, $V_2$ is the 6-dimensional vector space spanned by the basis $\{x^2, y^2, z^2, xy, xz, yz\}$. Let $C_2$ be the projective space associated with the vector space $V_2$, so that $C_2$ can be identified with $\mathbb{P}^5$. Note that we have a natural action of the group $\GL(3)$ on $V_2$. This induces an action of the projective linear group $\PGL(3)$ on the Grassmannian $\Gr(3, V_2)$, which parametrizes the three-dimensional subspaces of $V_2$ (We call an element of $\Gr(3, V_2)$ a ``net of conics"). To be precise, given an element $\langle v_1, v_2, v_3 \rangle \in \Gr(3, V_2)$ and $T \in \PGL(3)$, fix $\sigma \in \GL(3)$ to be a representative element of $T$, and define $T \cdot \langle v_1, v_2, v_3 \rangle = \langle \sigma(v_1), \sigma(v_2), \sigma(v_3) \rangle$. Before we embark on our analysis, it is first imperative to classify the orbits of this action of $\PGL(3)$ on $\Gr(3, V_2)$. \\ \\ Note that the variety $\Gr(3, V_2)$ has dimension $3(6 - 3) = 9$, while $\PGL(3)$ has dimension $3^2 - 1 = 8$. Therefore, we expect to find a one-parameter family of orbits of the $\PGL(3)$ action on $\Gr(3, V_2)$. Now, in order to classify the orbits, we need to make a few elementary constructions. Let $S_4$ denote the cubic hypersurface of degenerate conics in $C_2$ i.e. $S_4 = \{[v] \in C_2: \exists b, l \in V_1 \text{ such that } v = l \cdot b\}$.  Given an element $W \in \Gr(3, V_2)$, let $P(W) \subset C_2$ denote the projective plane associated to $W$. Then, define $\Gamma(W) := P(W) \cap S_4$. Since $\Gamma(W)$ is the intersection of a projective plane with a cubic hypersurface, it follows that $\Gamma(W)$ is generically a cubic plane curve. Moreover, let $D_2$ denote the set of double lines in $C_2$ i.e. it is the set of singular conics whose equation is the square of a linear form. \\ \\ Before moving on, we must define the Jacobian net of a smooth cubic plane curve. Let $\phi$ be a smooth cubic plane curve in 3 variables (say $x, y$ and $z$). Let $\phi_x, \phi_y,$ and $\phi_z$ denote the partial derivatives of $\phi$ with respect to the 3 variables. Then, $\phi_x, \phi_y$ and $\phi_z$ are all homogeneous conics in 3 variables. Then, the Jacobian net $J\phi$ is defined to be the vector subspace of $V_2$ spanned by $\phi_x, \phi_y, $ and $\phi_z$. \\ \\ Then, the following 2 theorems from \cite{AbdallahEmsalemIarrobino2021} classify the orbits of the action of $\PGL(3)$ on $\Gr(3, V_2)$ by considering the cases where $\Gamma(W)$ is smooth and where $\Gamma(W)$ is singular. 
\begin{theorem}
For every $W \in \Gr(3, V_2)$ such that $\Gamma(W)$ is a smooth cubic, there exists a unique smooth plane cubic $\phi$ such that $W = J\phi$. Moreover, the $\PGL(3)$-orbit of $W$ is completely determined by the $j$-invariant of the corresponding cubic $\phi$.
\end{theorem}

The above theorem gives us a one-parameter family of orbits corresponding to the case where $\Gamma(W)$ is smooth, as desired. Now, the theorem below classifies the orbits corresponding to the case where $\Gamma(W)$ is singular: 

\begin{theorem}
For every $W \in \Gr(3, V_2)$ such that $\Gamma(W)$ is not a smooth cubic, the $\PGL(3)$-orbit of $W$ is completely determined by the linear isomorphism class of the triple of projective algebraic sets \[ [D_2 \cap P(W), \Gamma(W) = S_4 \cap P(W), P(W)] \] viewed as schemes. 
\end{theorem}

We can use the well-known classification of cubics in a projective plane over an algebraically closed field to give a general classification of cases for $\Gamma$. This is done in Section 4.0 of \cite{AbdallahEmsalemIarrobino2021}. Then, the authors of \cite{AbdallahEmsalemIarrobino2021} use this classification along with Theorem 1.2 to classify the $\PGL(3)$-orbits corresponding to the case where $\Gamma(W)$ is not a smooth cubic. In particular, they show that this case gives rise to exactly 14 distinct orbits. Their argument can be found in Section 4.2 of \cite{AbdallahEmsalemIarrobino2021}; we omit the argument here for the sake of brevity. \\ \\ In this paper, we will compute the classes of the codimension 1 $\PGL(3)$-orbit closures in the Chow ring of the Grassmannian $\Gr(3, V_2)$. The analogous simpler problem for $\Gr(2, V_2)$ was completely resolved in \cite{Goel2023}. As shown in Table 1 in \cite{AbdallahEmsalemIarrobino2021}, every orbit in the one-parameter family of orbits corresponding to the case where $\Gamma(W)$ is smooth has codimension 1. Moreover, it also follows from Table 1 in \cite{AbdallahEmsalemIarrobino2021} that there are exactly two codimension 1 orbits corresponding to the case where $\Gamma(W)$ is not a smooth cubic curve. In Section 3 of this paper, we prove that every codimension 1 scheme-theoretic orbit closure has class $4\sigma_1$ in the Chow ring of the Grassmannian $\Gr(3, V_2)$. 

\begin{theorem}
Suppose that $O$ is a codimension 1 orbit of the $\PGL(3)$ action on $\Gr(3, V_2)$. Let $\overline{O}$ denote the scheme-theoretic closure of $O$. Then, we have $[\overline{O}] = 4\sigma_1$.
\end{theorem}

Finally, in section 4, we try to account for the potential non-reduced structure of the scheme-theoretic orbit closures and compute the classes of the codimension 1 set-theoretic orbit closures in the Chow ring of $\Gr(3, V_2)$. 

\section{Acknowledgements}

This work was funded by the Harvard College Research Program (HCRP). The author of this paper would like to thank Professor Joseph Harris for his invaluable guidance, as well as Professor Anand Patel and Dhruv Goel for many helpful suggestions. 

\section{Computing the Chow Ring Classes of the \\ Codimension 1 orbit closures}

Fix an element $W \in \Gr(3, V_2)$ such that $\Gamma(W)$ is smooth. By Theorem 1.1, there exists a unique smooth plane cubic $\phi_W$ such that $W = J\phi_W$. First, we will give an equation relating the $j$-invariant of $\phi_W$ with the $j$-invariant of $\Gamma(W)$. Note that the cubic curve $\Gamma(W)$ is isomorphic to the Hessian $H(\phi_W)$ of the smooth cubic $\phi_W$. Moreover, it follows from Proposition 2.2 in \cite{Mula2025} that \[ j(H(\phi_W)) = \frac{(6912 - j(\phi_W))^3}{27 (j(\phi_W))^2}. \] Thus, we get the following result 

\begin{theorem}
Fix $W \in \Gr(3, V_2)$ such that $\Gamma(W)$ is smooth. Let $\phi$ be the unique smooth plane cubic such that $W = J\phi$. Then, we have \[ j(\Gamma(W)) = \frac{(6912 - j(\phi))^3}{27 (j(\phi))^2}. \]
\end{theorem}

Now, consider the rational map $p: \Gr(3, V_2) \dashrightarrow \mathbb{P}^1$ given by $p(W) = j(P(W) \cap S_4)$. Concretely, $p$ sends a vector subspace $W \in \Gr(3, V_2)$ to the $j$-invariant of the cubic plane curve obtained by intersecting the projective plane $P(W)$ with the cubic hypersurface of degenerate conics. Let $U$ be the domain of definition of the rational map $p$. Then, $U$ is the largest open set on which $p$ is regular. Let $p_U: U \to \mathbb{P}^1$ denote the corresponding regular map. 
\begin{lemma} 
The regular map $p_U: U \to \mathbb{P}^1$ is surjective. 
\end{lemma} 
\begin{proof} 
Fix an element $j_0 \in \mathbb{P}^1$ with $j_0 \neq \infty$. Then, since we are working over an algebraically closed field of characteristic zero, there exists $j_1$ such that $j_0 (27j_1^2) = (6912 - j_1)^3 $. Note that $j_1 \neq 0$ since $j_0 \neq \infty$. Thus, let $\phi_0$ be a smooth plane cubic with $j$-invariant $j_1$. Since $j_1 \neq 0$, it follows that $\phi_0$ is not isomorphic to $X^3 + Y^3 + Z^3$. Then, it follows from Proposition 4.3 in \cite{AbdallahEmsalemIarrobino2021} that $\Gamma(J\phi_0) = P(J\phi_0) \cap S_4$ is a smooth cubic curve. Then, it follows from Theorem 3.1 that $j(\Gamma(J\phi)) = j_0$. Consequently, we have that $p(J\phi) = p_U(J \phi) = j_0$. Hence, every $j_0 \neq \infty$ is contained in the image of $p_U$. \\ \\ In order to show that $p_U$ is surjective, we still need to show that $j_0 = \infty$ is contained in the image of $p_U$. To this end, define $H_0 = \langle xy, x^2 + yz, y^2 + xz \rangle \in \Gr(3, V_2)$. Then, every conic in $H_0$ can be written in the form $axy + b(x^2 + yz) + c(y^2 + xz)$. The $3 \times 3$ symmetric matrix corresponding to the conic $axy + b(x^2 + yz) + c(y^2 + xz)$ is \[ a \begin{pmatrix}
0 & 1/2 & 0 \\
1/2 & 0 & 0 \\
0 & 0 & 0 
\end{pmatrix} + b \begin{pmatrix}
1 & 0 & 0 \\
0 & 0 & 1/2 \\
0 & 1/2 & 0 
\end{pmatrix} + c \begin{pmatrix}
0 & 0 & 1/2 \\
0 & 1 & 0 \\
1/2 & 0 & 0 
\end{pmatrix}   \]  \[ = \begin{pmatrix}
b & a/2 & c/2 \\
a/2 & c & b/2 \\
c/2 & b/2 & 0 
\end{pmatrix} .\]

Therefore, the conic $axy + b(x^2 + yz) + c(y^2 + xz)$ is singular if and only if \[ \det\begin{pmatrix}
b & a/2 & c/2 \\
a/2 & c & b/2 \\
c/2 & b/2 & 0 
\end{pmatrix} = - \frac{b^3}{4} - \frac{c^3}{4} + \frac{abc}{4} = 0 .\]

Equivalently, the conic $axy + b(x^2 + yz) + c(y^2 + xz)$ is singular if and only if $-b^3 - c^3 + abc = 0$. Therefore, $\Gamma(H_0)$ is isomorphic to the cubic plane curve $-b^3 -c^3 +abc = 0$. Note that this cubic is singular. In particular, this cubic has an ordinary double point (or a node) at the point $[1:0:0]$, and it has no other singularities. Thus, the $j$-invariant of this cubic curve is defined and is equal to $\infty$. In short, $j(\Gamma(H_0)) = \infty$. This implies that $p_U(H_0) = p(H_0) = \infty$. Therefore, $j_0 = \infty$ is contained in the image of $p_U$. This completes our argument that $p_U$ is surjective. \end{proof} 

Since $p_U$ is surjective, it is in particular dominant. Since $U$ is a non-empty open set of $\Gr(3, V_2)$, we have $\dim(U) = \dim(\Gr(3, V_2)) = 9$. Since $p_U$ is dominant, we have that a generic fiber of $p_U$ has codimension 1 in $U$, and that every fiber of $p_U$ has codimension at most 1 in $U$ [Hartshorne, Exercise 3.22]. Moreover, since the map $p_U$ is non-constant, none of the fibers of $p_U$ are 9-dimensional. Therefore, every fiber of $p_U$ has codimension 1 in $U$. This implies that the closure of $p_U^{-1}(j)$ in $\Gr(3, V_2)$ has codimension 1 in $\Gr(3, V_2)$ for all $j \in \mathbb{P}^1$. \\ \\ Let $A(\Gr(3, V_2))$ denote the Chow ring of $\Gr(3, V_2)$. Recall that $\overline{p_U^{-1}(j)}$ has codimension 1 in $\Gr(3, V_2)$ for all $j \in \mathbb{P}^1$. Therefore, for all $j \in \mathbb{P}^1$, there exists $a_j \in \mathbb{Z}_{\geq 0}$ such that $[\overline{p_U^{-1}(j)}] = a_j \sigma_1$. Then, in order to compute $a_j$, we can intersect both sides of the previous equation with the complementary Schubert cycle $\sigma_{3, 3, 2} \in A(\Gr(3, V_2))$. Note that $\sigma_{3,3,2}$ is the set of 3-planes contained in a given 4-plane $\Lambda_4 \subset V_2$ and that contain a given 2-plane $\Lambda_2 \subset V_2$. Recall that $C_2$ is the projective space associated to $V_2$. Then, by Kleiman's transversality theorem, it follows that $a_j$ is the answer to the following enumerative problem: Given a general projective 3-plane $\Lambda \subset C_2 \cong \mathbb{P}^5$ and a general line $L \subset C_2 \cong \mathbb{P}^5$, how many 2-planes $H \subset C_2 \cong \mathbb{P}^5$ containing $L$ and lying in $\Lambda$ intersect with $S_4$ (the cubic hypersurface of degenerate conics) to give a cubic plane curve with $j$-invariant equal to $j$. \\ \\ Now, we will compute $a_{\infty}$. \begin{theorem} 
We have $a_{\infty} = 8$.
\end{theorem} 
\begin{proof}
Consider the closed codimension 1 subvariety $\overline{p_U^{-1}(\infty)}$. Recall that $p_U^{-1}(\infty)$ is the set of 3-planes $H \in \Gr(3, V_2)$ such that $P(H) \cap S_4$ is a cubic plane curve with $j$-invariant $\infty$. In particular, it follows that $p_U^{-1}(\infty)$ contains all 3-planes $H \in G(3, V_2)$ such that $P(H) \cap S_4$ is a nodal cubic. Moreover, it follows from Table 1 in \cite{AbdallahEmsalemIarrobino2021} that the locus in $\Gr(3, V_2)$ of 3-planes $H$ such that $P(H) \cap S_4$ has a singularity other than a node has codimension $\geq 2$. Therefore, $a_{\infty}$ is the answer to the following enumerative problem: Given a general projective 3-plane $\Lambda \subset C_2$ and a general line $L \subset C_2$, how many projective 2-planes containing $L$ and lying in $\Lambda$ intersect with $S_4$ to give a singular cubic. \\ \\ Note that the singular locus of the cubic hypersurface $S_4 \subset C_2 \cong \mathbb{P}^5$ of degenerate conics is exactly the Veronese surface $D_2$, which is the locus of double lines in the $\mathbb{P}^5$ of conics. Since $D_2 \cong \mathbb{P}^2$ and $\deg(D_2) = 4$, the intersection of $D_2$ with a general projective 3-plane consists of four distinct points. Therefore, the singular locus of the cubic hypersurface $S_4 \cap \Lambda \subset \Lambda \cong \mathbb{P}^3$ consists of exactly four points. Let $X := S_4 \cap \Lambda$. Observe that at each point $p$ in the singular locus of $X$, the projective tangent cone $\mathbb{T}C_p X$ is clearly a smooth quadric, so each point in the singular locus of $X$ is an ordinary double point. Therefore, $X = S_4 \cap \Lambda \subset \Lambda \cong \mathbb{P}^3$ is a cubic hypersurface with four ordinary double points that is smooth everywhere else. \\ \\ Note that if a projective 2-plane $H$ lying in $\Lambda$ intersects with $X$ (or $S_4$) to give a singular cubic, then either $H$ contains one of the four ordinary double points of $X$, or $H$ is tangent to $X$ at a smooth point of $X$. Note that given a general line $L \subset \Lambda \cong \mathbb{P}^3$ and one of the four ordinary double points $p \in X$, there is exactly one 2-plane that contains $L$ and $p$. Therefore, there are exactly four projective 2-planes $H$ containing $L$ and lying in $\Lambda$ that contain one of the four ordinary double points of $X$. \\ \\ Furthermore, observe that the number of projective 2-planes $H \subset \Lambda \cong \mathbb{P}^3$ containing a general line $L$ that are tangent to $X$ at a smooth point of $X$ is simply the degree of the dual hypersurface $X^* \subset \mathbb{P}^{3*}$. Hence, now we will compute the degree of the dual hypersurface $X^* \subset \mathbb{P}^{3*}$. Observe that $X$ is simply the Cayley cubic surface. It is a classical fact that the dual of the Cayley cubic surface is a Steiner quartic surface, as shown in Section 9.2 of \cite{Dolgachev}. In particular, it follows that the degree of the dual hypersurface $X^* \subset \mathbb{P}^{3*}$ is 4. Below, we present another modern proof of this result.  

\begin{lemma}
The degree of the dual hypersurface $X^* \subset \mathbb{P}^{3*}$ is 4. 
\end{lemma}

\begin{proof}
Let $B$ denote the blow-up of $\mathbb{P}^3$ at the four ordinary double points $p_1, p_2, p_3, p_4 \in X$. Now, let $A(B)$ denote the Chow ring of $B$. Let $A^k(B)$ denote the group of codimension $k$ cycles of $B$. Let $E_1, E_2, E_3,$ and $E_4$ be the exceptional divisors associated to the points $p_1, p_2, p_3$ and $p_4$ i.e. $E_i = \pi^{-1}(p_i)$ for all $1 \leq i \leq 4$, where $\pi: B \to \mathbb{P}^3$ is the morphism associated with the blow-up. Let $\Lambda_k$ be the preimage $\pi^{-1}(\Lambda_k')$ of a $(3 - k)$-plane $\Lambda_k' \cong \mathbb{P}^2 \subset \mathbb{P}^3$ that does not contain the points $p_i$. For each $1 \leq i \leq 4$, let $\Gamma_{k, i}$ be a $(3 - k)$-plane contained in the exceptional divisor $E_i \cong \mathbb{P}^2$. Then, it follows from the well-known characterization of the chow ring of a blow-up at finitely many points that $A^k(B) = \mathbb{Z} \langle [\Lambda_k], [\Gamma_{k, 1}], \dots, [\Gamma_{k, 4}]\rangle$, and the multiplication relations of the Chow ring are: $$\Lambda_k \cdot \Lambda_l = \Lambda_{k + l}$$ $$\Gamma_{k, i} \cdot \Gamma_{l, j} = 0 \text{ for all } j \neq i$$   $$\Lambda_k \cdot \Gamma_{l, i} = 0$$ $$\Gamma_{k, i} \cdot \Gamma_{l, i} = - \Gamma_{k + l, i}$$ Let $\lambda = [\Lambda_1]$ and $ e_i = [E_i]$ for all $1 \leq i \leq 4$. Then, we have that $A^1(B) = \mathbb{Z}\langle \lambda, e_1, \dots, e_4 \rangle $. Now, let $\tilde{X}$ denote the proper transform of $X$ in $B$ i.e. $\tilde{X}$ is the closure of the preimage $\pi^{-1}(X \backslash \{p_1, p_2, p_3, p_4\})$. Then, observe that $[\tilde{X}] \in A^1(X)$, which implies that \[[\tilde{X}] = d\lambda + m_1 e_1 + m_2 e_2 + m_3 e_3 + m_4 e_4\] for some $d, m_i \in \mathbb{Z}$. Now, fix $1 \leq i \leq 4$. Consider the inclusion map $j: E_i \to B$. It follows from the Chow ring relations above that $j^*(\lambda) = 0$, $j^*(e_i) = -l$, where $l$ is the class of a line in $E_i \cong \mathbb{P}^2$, and $j^*(e_l) = 0$ for $l \neq i$. Hence, it follows that \[ j^*([\tilde{X}]) = - m_i l \in A^1(E_i).\] Now, note that the projectivized tangent cone $\mathbb{T}C_{p_i}X$ is exactly the intersection of the proper transform $\tilde{X}$ with the exceptional divisor $E_i$. Moreover, since $p_i$ is an ordinary double point, we have that $\mathbb{T}C_{p_i}X$ is a smooth quadric i.e. it has degree 2. Therefore, $-m_i l = 2l$, which implies that $m_i = -2$. Since $i$ was chosen arbitrarily, it follows that $m_i = -2$ for all $1 \leq i \leq 4$. \\ \\ Now, consider the inclusion map $j_0: \Lambda_1 \to B$. Observe that $j_0^*(\lambda) = l$, where $l$ is the class of a line in $\Lambda_1 \cong \mathbb{P}^2$, and $j_0^*(e_i) = 0$ for all $i$. Thus, $j_0^{*}([\tilde{X}]) = dl \in A^1(\Lambda_1)$. Since the hypersurface $X$ has degree 3 in $\mathbb{P}^{3}$, it immediately follows that $d = 3$. Hence, we have shown that \[ [\tilde{X}] = 3\lambda - 2e_1 - 2e_2 - 2e_3 - 2e_4 .\] Let $F$ be the polynomial defining the hypersurface $X$ i.e. $X$ is the zero locus of the homogeneous polynomial $F(Z_0, \dots, Z_3)$. Now, consider the map $G: \mathbb{P}^3 - \{p_1, p_2, p_3, p_4\} \to \mathbb{P}^{3*} $ given by \[ G(p) = \left[\frac{\partial F}{\partial Z_0}(p), \dots, \frac{\partial F}{\partial Z_4}(p) \right] . \] In particular, $G$ sends a smooth point $p \in X$ to its tangent hyperplane $\mathbb{T}_p X$. It is easy to see that this map extends to a regular morphism $G_0: B \to \mathbb{P}^{3*}$. Let $h \in A^1(\mathbb{P}^{3*})$ be the class of a general hyperplane $H \subset \mathbb{P}^{3*}$. It is clear that $G_0^*(h) = c \lambda - e_1 - e_2 - e_3 - e_4$ for some $c \in \mathbb{Z}$. Moreover, since the partial derivatives of $F$ are all polynomials of degree $3 - 1 = 2$, it follows that $c = 2$. Therefore, \[ G_0^*(h) = 2\lambda - e_1 - e_2 - e_3 - e_4  .\] It follows from Bertini's theorem that $2$ general hypersurfaces given by the pullback of the hyperplane class $h$ intersect transversely in $B$. The degree of the dual hypersurface $X^*$ is thus given by \[ \deg([\tilde{X}] (G_0^*(h))^2) = \deg((2\lambda - e_1 - e_2 - e_3 - e_4)^2(3\lambda - 2e_1 - 2e_2 - 2e_3 - 2e_4)).\] By the relations of the Chow ring $A(B)$, the above simplifies to \[ \deg(12 \lambda^3 - 2 e_1^3 - 2e_2^3 - 2e_3^3 - 2e_4^3) = 12 - 2(4) = 4 . \] Thus, the degree of the dual hypersurface $X^* \subset \mathbb{P}^{3*}$ is 4, as desired. 
\end{proof}
Now, we return to the proof of Theorem 3.3. It follows from Lemma 2.4 that given a general projective 3-plane $\Lambda \subset C_2$ and a general line $L \subset C_2$, there are exactly four projective 2-planes containing $L$ and lying in $\Lambda$ that are tangent to $X = S_4 \cap \Lambda$ at a smooth point of $X$. Moreover, recall that there are exactly four projective 2-planes containing $L$ and lying in $\Lambda$ that contain one of the four ordinary double points of $X$. Therefore, there are $4 + 4 = 8$ projective 2-planes containing $L$ and lying in $\Lambda$ that intersect with $S_4$ to give a singular cubic. In particular, we have $a_{\infty} = 8$, as desired.  
\end{proof}
Let $O_1'$ denote the set of subspaces $W \in \Gr(3, V_2)$ such that $\Gamma(W)$ is a nodal cubic whose node lies on the Veronese surface. Let $O_2'$ denote the set of subspaces $W \in \Gr(3, V_2)$ such that $\Gamma(W)$ is a nodal cubic whose node does not lie on the Veronese surface. Now, recall from \cite{AbdallahEmsalemIarrobino2021} that $O_1'$ and $O_2'$ are the two codimension 1 orbits corresponding to the case where $\Gamma(W)$ is not a smooth cubic curve. Write $[\overline{O_1'}] = b_1 \sigma_1$ and $[\overline{O_2'}] = b_2 \sigma_1$, where $b_1, b_2 \in \mathbb{Z}_{> 0}$. It follows from Table 1 in \cite{AbdallahEmsalemIarrobino2021} that the locus of 3-planes $W \in \Gr(3, V_2)$ such that $\Gamma(W)$ has a singularity other than a node and $\Gamma(W)$ contains a point on the Veronese surface has codimension $\geq 2$. Therefore, intersecting with the Schubert cycle $\sigma_{3, 3, 2}$, we see that $b_1$ is the number of projective 2-planes $H \subset C_2$ containing a fixed general line $L \subset C_2$ and lying in a fixed general 3-plane $\Lambda$ such that $H \cap S_4$ is singular and contains a point on the Veronese surface. Then, it follows from our proof of Theorem 3.3 that $b_1 = 4$. Similarly, it follows from the proof of Theorem 3.3 (specifically Lemma 2.4) that $b_2 = 4$. Therefore, we get 

\begin{theorem}
We have $[\overline{O_1'}] = [\overline{O_2'}] = 4 \sigma_1$.
\end{theorem}

Now, let us return to finding the classes of the orbit closures corresponding to smooth $\Gamma(W)$. Let $F_b$ denote the \textbf{scheme-theoretic fiber} of the map $p: \Gr(3, V_2) \dashrightarrow \mathbb{P}^1$ over $b \in \mathbb{P}^1$. Concretely, $F_b$ is the closed subscheme of $\Gr(3, V_2)$ given by the fiber product $$F_b = \overline{\Gamma_p} \times_{\mathbb{P}^1} \Spec(\kappa(b))$$ where $\overline{\Gamma_p}$ is the closure of the graph of the rational map $p: \Gr(3, V_2) \dashrightarrow \mathbb{P}^1$ in $\Gr(3, V_2) \times \mathbb{P}^1$ and $\kappa(b)$ is the residue field of the point $b \in \mathbb{P}^1$. Recall that the set-theoretic fiber $p_U^{-1}(\infty)$ is the set of planes $W \in \Gr(3, V_2)$ such that $\Gamma(W)$ is a nodal cubic. Moreover, since $\PGL(3)$ is an irreducible group and since $O_1'$ and $O_2'$ are orbits of the $\PGL(3)$-action on $\Gr(3, V_2)$, it follows that the orbit closures $\overline{O_1'}$ and $\overline{O_2'}$ are irreducible. In particular, we have that $\overline{O_1'}$ and $\overline{O_2'}$ are the codimension 1 irreducible components of the scheme-theoretic fiber $F_{\infty}$. Therefore, we may write $[F_{\infty}] = l_1 [\overline{O_1'}] + l_2 [\overline{O_2'}]$ for some $l_1, l_2 \in \mathbb{Z}_{> 0}$. \\ \\ Let $H$ be a general point of $O_1'$. Fix a general pencil of 3-planes $\{H_t\}_{t \in \mathbb{P}^1}$ such that $H_0 = H$. Let $\{C_t\}_{t \in \mathbb{P}^1}$ be the family of cubics given by $C_t = \Gamma(H_t)$. Then, $C_0$ is a nodal cubic whose node lies on the Veronese surface. Therefore, locally around the node of $C_0$, the family of cubics has equation $xy - t^2$. This implies that $p$ has a double pole along $\overline{O_1'}$. Hence, $l_1 = 2$. An analogous argument shows that $p$ has a simple pole along $\overline{O_2'}$, which implies that $l_2 = 1$. Therefore, we have \[ [F_{\infty}] = 2[\overline{O_1'}] + [\overline{O_2'}] = 12\sigma_1 \] by Theorem 3.5. Observe that the scheme-theoretic fibers of the map $p: \Gr(3, V_2) \to \mathbb{P}^1$ are all rationally equivalent to each other by definition. Therefore, it follows that $[F_b] = [F_{\infty}] = 12\sigma_1$ for all $b \in \mathbb{P}^1$. 
\begin{theorem}
We have $[F_b] = 12\sigma_1$ for all $b \in \mathbb{P}^1$. 
\end{theorem}

Having computed the Chow ring classes of the scheme-theoretic fibers of the rational map $p$, we are now ready to compute the Chow ring classes of the orbit closures of the $\PGL(3)$ action on $\Gr(3, V_2)$ corresponding to the case where $\Gamma(W)$ is smooth. Recall from Proposition 4.3 in \cite{AbdallahEmsalemIarrobino2021} that: 

\begin{itemize}
    \item Given an element $W \in \Gr(3, V_2)$ such that $\Gamma(W)$ is a smooth cubic, there exists a unique smooth cubic plane curve $\phi$, not isomorphic to $X^3 + Y^3 + Z^3$, such that $W = J\phi$
    \item  Conversely, given a smooth cubic $\phi$ that is not isomorphic to $X^3 + Y^3 + Z^3$, we have that $\Gamma(J\phi)$ is a smooth cubic.
    \item The $\PGL(3)$-orbit of $W \in \Gr(3, V_2)$ is determined by the $j$-invariant of the corresponding smooth cubic $\phi$.  
\end{itemize}

Let $O_b$ denote the orbit consisting of $W \in \Gr(3, V_2)$ such that $\Gamma(W)$ is a smooth cubic and the corresponding smooth cubic $\phi$ has $j$-invariant $b$. Let $\overline{O_b}$ denote the \textbf{scheme-theoretic} orbit closure. Now, recall from Theorem 3.1 that given $W \in \Gr(3, V_2)$ such that $\Gamma(W)$ is smooth, we have \[ j(\Gamma(W)) = \frac{(6912 - j(\phi))^3}{27(j(\phi))^2}  .\] Observe that each of the orbits $O_b$ are irreducible (since $\PGL(3)$ is irreducible) and have codimension 1 \cite{AbdallahEmsalemIarrobino2021}. Therefore, the equation in Theorem 3.1 implies that for all $b \neq \infty$, \[ [F_b] = \sum_{i} m_i [\overline{O_{\alpha_i}}] \] for some $m_i \in \mathbb{Z}_{> 0}$, where the $\alpha_i$ are the roots of the polynomial $f_b(j) := (6912 - j)^3 - 27bj^2$. Observe that $m_i$ is simply the multiplicity of the root $\alpha_i$ in the polynomial $f_b(j)$. \begin{lemma} 
We have that $f_b(j)$ only has a repeated root for $b = 0$ and $b = 1728$.
\end{lemma}

\begin{proof}
Fix $b$. Suppose that $f_b(j)$ has a repeated root. Then, $f_b(j) = (6912 - j)^3 - 27bj^2$ and $f_b'(j) = -3(6912 - j)^2 - 54bj$ have a common zero $x$. Since $f_b'(x) = 0$, we get \[ b = -\frac{(6912 - x)^2}{18x} .\] Substituting the above in $f_b(x) = 0$ yields \[ (6912 - x)^3 = -\frac{3}{2} (6912 - x)^2x\] which is equivalent to \[ (x - 6912)^2 (x + 13824) = 0 .\] Thus, $x = 6912$ or $x = -13824$. If $x = 6912$, then $b = 0$. If $x = -13824$, then \[ b = - \frac{(6912 + 13824)^2}{18(-13824)} = 1728 .\] Hence, we have either $b = 0$ or $b = 1728$, as desired. Moreover, it is easy to verify that if $b = 0$, then $f_b(j)$ has a triple root at $6912$, and if $b = 1728$, then $f_b(j)$ has a double root at $-13824$ and a simple root at $1728$. 
\end{proof}

In particular, Lemma 2.7 implies that $f_b(j)$ has three simple roots for $b \neq 0, 1728$. Thus, for all $b \neq 0, 1728$ we have \[ [F_b] = [\overline{O_{\alpha_1}}] + [\overline{O_{\alpha_2}}] + [\overline{O_{\alpha_3}}] \] where $\alpha_1, \alpha_2, $ and $\alpha_3$ are the three distinct roots of $f_b(j)$. Moreover, we also have \[ [F_0] = 3[\overline{O_{6912}}] \] and \[ [F_{1728}] = [\overline{O_{1728}}] + 2[\overline{O_{-13824}}] .\] Now, we are ready to prove the main result. \begin{theorem}
For all $b \in \mathbb{P}^1 \backslash \{0\}$, we have $[\overline{O_b}] = 4\sigma_1$.
\end{theorem}

\begin{proof}
Consider the map $j: \mathbb{P}^1 - \{0\} \to \mathbb{P}^1$ given by $j(x) = \frac{(6912 - x)^3}{27x^2}$. Let $\Gamma_j$ denote the graph of the map $j$. Then, $\Gamma_j$ is birational to $\mathbb{P}^1 - \{0\}$ via the first projection map, and is thus irreducible. Since $\Gamma_j$ parametrizes the orbits $O_b$ and since it is irreducible, it follows that $[\overline{O_{b_1}}] = [\overline{O_{b_2}}]$ for all $b_1, b_2 \neq 0$. In particular, fixing $q \neq 0, 1728$, we get \[ 12\sigma_1 = [F_q] = [\overline{O_{\alpha_1}}] + [\overline{O_{\alpha_2}}] + [\overline{O_{\alpha_3}}] = 3 [\overline{O_{\alpha_1}}] \] which implies that $[\overline{O_{\alpha_1}}] = 4\sigma_1$. This in turn implies that $[\overline{O_b}] = 4\sigma_1$ for all $b \neq 0$, as desired.    
\end{proof}

Now, we will give an alternative proof of Theorem 3.8 for $b \neq 0, 6912, 1728, -13824$. 

\begin{proof}
In this proof, we assume that $k = \mathbb{C}$. Fix $b \neq 0, 6912, 1728, -13824$. Now, observe that the map $j: \mathbb{P}^1 - \{0, 6912, 1728, -13824, \infty\} \to \mathbb{P}^1 - \{0, 1728, \infty\}$ given by $j(x) = \frac{(6912 - x)^3}{27x^2}$ is a degree 3 covering map. Moreover, recall that $\mathbb{P}^1$ is the Riemann sphere, so $X := \mathbb{P}^1 - \{0, 6912, 1728, -13824, \infty\}$ is a connected space. Now, let $a = j(b)$ and let $j^{-1}(a) = \{b, b_1, b_2\}$. Then, since $X$ is connected, it follows that the monodromy action of $\pi_1(\mathbb{P}^1 - \{0, 1728, \infty\}, a)$ on $j^{-1}(a)$ is transitive. Therefore, the orbits $O_{b}, O_{b_1}, O_{b_2}$ also experience transitive monodromy, which implies that their cycle classes are homologous. Therefore, \[ 12\sigma_1 = [F_a] = [O_b] + [O_{b_1}] + [O_{b_2}] = 3[O_b] \] which implies that $[\overline{O_b}] = 4\sigma_1$, as desired.    
\end{proof}

Finally, observe that by combining Theorem 3.5 and Theorem 3.8, we immediately get Theorem 1.3, as desired. 

\section{Further work}

Note that Theorem 3.8 tells us that the \textbf{scheme-theoretic} orbit closure has class $4\sigma_1$ in the Chow ring of $\Gr(3, V_2)$. However, the scheme-theoretic orbit closure may be non-reduced (i.e. it may appear with multiplicity), and so we need to account for this in order to compute the Chow ring classes of the set-theoretic orbit closures. To this end, we will try to compute the stabilizers of elements $W \in \Gr(3, V_2)$ such that $\Gamma(W)$ is smooth. \\ \\ Fix $\lambda \neq -1, -\omega, -\omega^2, 0, 2, 2\omega, 2\omega^2$. Then, the cubic $\phi_{\lambda} =  X^3 + Y^3 + Z^3 + 3\lambda XYZ$ is a smooth cubic with $j$-invariant $j(\lambda)$\footnote{According to \cite{Frium2002}, we may take $$j(\lambda) = \frac{\lambda^3(\lambda^3 - 8)^3}{27(\lambda + 1)^3(\lambda + \omega)^3(\lambda + \omega^2)^3}$$ } that is not isomorphic to the cubic  $X^3 + Y^3 + Z^3 = 0$.  Consider the orbit $O_{j(\lambda)}$. Then, $J \phi_{\lambda} = \langle X^2 + \lambda YZ, Y^2 + \lambda XZ, Z^2 + \lambda XY \rangle \in O_{j(\lambda)}$. Let $W_{\lambda} = J\phi_{\lambda}$. Now, suppose that $\sigma' \in \stab(W_{\lambda})$. Let $\sigma \in \GL(3)$ be a representative element of $\sigma'$ with matrix representation 

\[ \begin{pmatrix}
a_{11} & a_{12} & a_{13} \\
a_{21} & a_{22} & a_{23} \\
a_{31} & a_{32} & a_{33} 
\end{pmatrix} . \]

Then, since $\sigma' \in \stab(W_{\lambda})$, it follows that $\sigma(X^2 + YZ), \sigma(Y^2 + ZX), \sigma(Z^2 + XY) \in W_{\lambda}$. In other words \begin{equation}
\begin{split}
  (a_{g(1)1}X &+ a_{g(1)2}Y + a_{g(1)3}Z) \\
  &\quad+ \lambda\bigl(a_{g(2)1}X + a_{g(2)2}Y + a_{g(2)3}Z\bigr)
            \bigl(a_{g(3)1}X + a_{g(3)2}Y + a_{g(3)3}Z\bigr)
  \;\in\;W_{\lambda}.
\end{split}
\end{equation} for all $g \in \langle (1 2 3)\rangle \subset S_3 $. Observe that given a homogeneous conic $$f(X, Y, Z) = f_1 X^2 + f_2YZ + f_3 Y^2 + f_4 ZX + f_5 Z^2 + f_6 XY$$ we have $f(X, Y, Z) \in W_{\lambda}$ if and only if $\lambda f_1 - f_2 = 0$, $\lambda f_3 - f_4 = 0$, and $\lambda f_5 - f_6 = 0$. Then, it follows from equation (1) above that \begin{equation}
\begin{aligned}
  2a_{g(1)h(2)}\,a_{g(1)h(3)}
  &\;+\;\lambda\bigl(a_{g(2)h(3)}\,a_{g(3)h(2)}
      \;+\;a_{g(2)h(2)}\,a_{g(3)h(3)}\bigr)\\
  &\;-\;\lambda\,a_{g(1)h(1)}^2
      \;-\;\lambda^2\,a_{g(2)h(1)}\,a_{g(3)h(1)}
  \;=\;0.
\end{aligned}
\end{equation} for all $g, h \in \langle (123) \rangle $. \\ \\ Let $A_g$ be the permutation matrix corresponding to the permutation $g \in S_3$, so $(A_g)_{ij} = 1$ if $j = g(i)$ and $(A_g)_{ij} = 0$ otherwise. Moreover, define $$B = \begin{pmatrix} 
1 & 0 & 0 \\ 
0 & \omega & 0 \\
0 & 0 & \omega^2

\end{pmatrix}$$ Now, fix $\lambda = 1$. Then, (2) gives us 9 degree 2 equations in 9 variables. Solving for the $a_{ij}$ using Mathematica yields: \[ \sigma' = [B^k A_{g}] \] for some $k \in \{0, 1, 2\}$ and $g \in S_3$ ($[B^k A_{g}]$ denotes the image of $B^k A_{g} \in \GL(3)$ under the quotient map $q: \GL(3) \to \PGL(3)$). Conversely, it is easy to verify that if $\sigma' = [B^k A_g]$ for some $k \in \{0, 1, 2\}$ and $g \in S_3$, then $\sigma' \in \stab(W_{\lambda})$. Thus we have $\stab(W_{\lambda}) = \{[B^k A_g]: k \in \{0,1, 2\}, g \in S_3\} $ when $\lambda = 1$. Analogous mathematica computations and arguments yield that $\stab(W_{\lambda}) = \{[B^k A_g]: k \in \{0,1, 2\}, g \in S_3\}$ for $\lambda = -7, -5, -3, 1, 3, 5, 7$. Let $g_1 = (12)$ and let $g_2 = (123)$. Then, let $a = [A_{g_2}], b = [B]$ and $c = [A_{g_1}]$. Then, we have \[ \{[B^k A_g]: k \in \{0,1, 2\}, g \in S_3\}\] \[ \cong \langle a, b, c | a^3 = b^3 = c^2 = 1, ab = ba, cbc = b^{-1}, cac = a^{-1}  \rangle \cong C_3 \rtimes_{\psi} S_3  \]  where the homomorphism $\psi: S_3 \to Aut(C_3) \cong C_2$ is given by the quotient map $S_3 \to S_3/C_3 \cong C_2$. Hence, we have that \[ \stab(W_{\lambda}) = \{[B^k A_g]: k \in \{0,1, 2\}, g \in S_3\} \cong C_3 \rtimes_{\psi} S_3 \] for $\lambda = -7, -5, -3, 1, 3, 5, 7$. This leads us to conjecture that \[ \stab(W_{\lambda}) = \{[B^k A_g]: k \in \{0,1, 2\}, g \in S_3\} \cong C_3 \rtimes_{\psi} S_3 \] is true for a general choice of $\lambda$. If this conjecture is true, then the set-theoretic closure of $O_{j(\lambda)}$ has class $4\sigma_1$ in the Chow ring of $\Gr(3, V_2)$ for all $\lambda$ such that $\stab(W_{\lambda}) = \{[B^k A_g]: k \in \{0,1, 2\}, g \in S_3\} \cong C_3 \rtimes_{\psi} S_3$. Moroever, the set-theoretic closure of $O_{j(\lambda)}$ would have class either $4\sigma_1, 2\sigma_1$ or $\sigma_1$ for all other values of $\lambda$.     

\section{Suggestions for further investigation}

This work opens several directions for further research. We list some of these directions below: 

\begin{itemize}
    \item The most natural next step of this paper would be to compute the Chow ring classes of the $\PGL(3)$ orbit closures that have codimension strictly greater than 1. By the work in \cite{AbdallahEmsalemIarrobino2021}, this corresponds to computing the Chow ring classes of the 12 remaining special orbits of the $\PGL(3)$ action on $\Gr(3, V_2)$.  
    \item One can also try to prove our conjecture in Section 4 that \[ \stab(W_{\lambda}) = \{[B^k A_g]: k \in \{0,1, 2\}, g \in S_3\} \cong C_3 \rtimes_{\psi} S_3 \] is true for a general choice of $\lambda$. Moreover, one could also try to determine the exceptional values of $\lambda$ (if any) for which the above statement does not hold.
    \item Observe that in this paper we consider the orbit closures of the $\PGL(3)$ action on $\Gr(3, V_2)$. However, one could instead compute the Chow ring classes of the orbit closures of the $\PGL(3)$ action on $\Gr(d, V_i)$ for other values of $d$ and $i$.  
\end{itemize}

\end{document}